\documentclass[leqno,12pt]{amsart} 
\usepackage{comment}
\usepackage{yhmath}
\usepackage{graphicx}
\usepackage[inline]{enumitem}
\usepackage{indentfirst,amscd}
\usepackage{amsthm,amsmath, amssymb, amsfonts, empheq,mathtools,dsfont}
\usepackage{upgreek}
\theoremstyle{plain} 
\newtheorem{theorem}{\indent\sc Theorem}[section]
\newtheorem{lemma}[theorem]{\indent\sc Lemma}

\newtheorem{proposition}[theorem]{\indent\sc Proposition}

\theoremstyle{definition} 
\newtheorem{definition}[theorem]{\indent\sc Definition}
\newtheorem{remark}[theorem]{\indent\sc Remark}
\newtheorem{example}[theorem]{\indent\sc Example}
\newtheorem{notation}[theorem]{\indent\sc Notation}

\usepackage{hyperref}
\hypersetup{
    colorlinks=true,
    linkcolor=blue,
    filecolor=magenta,
    urlcolor=cyan,
}
\usepackage[capitalize]{cleveref}

\DeclareMathOperator{\QQ}{\mathbb{Q}}
\DeclareMathOperator{\ZZ}{\mathbb{Z}}
\DeclareMathOperator{\CC}{\mathbb{C}}
\DeclareMathOperator{\OO}{\mathcal{O}}
\DeclareMathOperator{\gal}{Gal}

\DeclareMathOperator{\ab}{\mathrm{ab}}
\DeclareMathOperator{\tors}{tors}
\DeclareMathOperator{\coker}{Coker}
\DeclareMathOperator{\Ker}{Ker}
\DeclareMathOperator{\order}{order}

\DeclareMathOperator{\Z}{Z}

\DeclareMathOperator{\D}{D}

\begin{document}

\title[TORSION OF ELLIPTIC CURVES WITH RATIONAL \(j\)-INVARIANT]{Torsion of elliptic curves with rational\\ \(j\)-invariant over the maximal elementary abelian 2-extension of \(\QQ\)} 

\author[Lucas Hamada]{Lucas Hamada} 

%
\address{
Department of Mathematics \endgraf
Institute of Science Tokyo \endgraf
2-12-1, O-okayama, Meguro-ku, Tokyo 152-8551 \endgraf
Japan
}
\email{lucas.h.r.hamada@gmail.com}

\maketitle

\begin{abstract}
 In this paper, we classify the possible torsion subgroup structures of elliptic curves defined over the compositum of all quadratic extensions of the rational number field, whose \(j\)-invariant is a rational number not equal to 0 or 1728.
\end{abstract}

\section{Introduction}\label{introduction} 
    Let \(E\) be an elliptic curve defined over \(\QQ\). B. Mazur proved in \cite{Mazur_modcur} that the torsion subgroup \(E(\QQ)_{\tors}\) is isomorphic to one of the following 15 groups:
    \begin{align*}
            \mathbb{Z}&\slash N_1 \mathbb{Z}, &\hspace{1cm}  &  N_1 = 1,\; \ldots, 10, 12,\\
            \mathbb{Z}\slash 2\mathbb{Z} &\oplus \mathbb{Z} \slash 2N_2 \mathbb{Z},   
            & \hspace{1cm} & N_2 = 1,\; \ldots, 4.
        \end{align*}
    Since then, there have been proven analogues for several other classes of elliptic curves. Let \(\Phi(d)\) denote the set of all isomorphism classes of torsion subgroups \(E(K)_{\tors}\), where \(K\) runs through all extensions of \(\QQ\) of degree \(d\) and \(E\) runs through all elliptic curves defined over \(K\). S. Kamienny in \cite{kamienny_1}, and M. Kenku and F. Momose in \cite{momose_kenku} described \(\Phi(2)\). The classification of \(\Phi(3)\) was completed by D. Jeon, C. H. Kim, and A. Schweizer in \cite{Jeon2004OnTT}, and  M. Derickx, A. Etropolski, J. Morrow, M. van Hoeij, and D. Zureick-Brown in \cite{10.2140/ant.2021.15.1837}.
    
   Similarly, let \(\Phi_{\QQ}(d)\) denote the set of all isomorphism classes of torsion subgroups \(E(K)_{\tors}\), where \(E\) runs through all elliptic curves defined over \(\QQ\) and \(K\) runs through all extensions of \(\QQ\) of degree \(d\). F. Najman determined \(\Phi_{\QQ}(2)\) and \(\Phi_{\QQ}(3)\) in \cite{najman_2016}. E. Gonz\'{a}lez-Jim\'{e}nez determined \(\Phi_{\QQ}(5)\) in \cite{GonzalezJimenez2016CompleteCO} and, jointly with F. Najman in \cite{GonzalezJimenez2016GrowthOT},  classified \(\Phi_{\QQ}(p)\) for all primes \(p \geq 7\). The same duo in \cite{GonzalezJimenez2016GrowthOT}, along with M. Chou in \cite{CHOU}, also determined \(\Phi_{\QQ}(4)\).

    Analogues of Mazur's theorem for infinite extensions of \(\QQ\) have also been established. Let \(\QQ(2^{\infty})\) denote the compositum of all quadratic extensions of \(\QQ\) (Assuming all fields contained in \(\CC\)). M. Laska and M. Lorenz in \cite{Laska_Lorenz}, and Y. Fujita in \cite{Fujita_1,Fujita_2}, classified, up to isomorphism, the possible torsion subgroups \(E(\QQ(2^{\infty}))_{\tors}\) for elliptic curves \(E\) defined over \(\QQ\):

    \begin{theorem}[M. Laska and M. Lorenz \cite{Laska_Lorenz} and Y. Fujita \cite{Fujita_1,Fujita_2}]\label{fujita_list}
    Let \(E\) be an elliptic curve defined over \(\QQ\). Then, the torsion subgroup \(E(\QQ(2^{\infty}))_{\tors}\) is isomorphic to one of the following groups\(:\)
    \begin{align*}
    \ZZ&/N_1\ZZ, &\hspace{1cm}  & N_1 = 1,3,5,7,9,15,\\
     \ZZ/2\ZZ &\oplus \ZZ/N_2\ZZ,     & \hspace{1cm}   & N_2 = 1,\; \ldots, 6, 8,\\
     \ZZ/3\ZZ &\oplus \ZZ/3\ZZ,\\
    \ZZ/4\ZZ &\oplus \ZZ/4N_4\ZZ, &\hspace{1cm}  & N_4 = 1,2,3,4,\\
    \ZZ/2N_5\ZZ &\oplus \ZZ/2N_5\ZZ, &\hspace{1cm}  & N_5 = 3,4.
        \end{align*}
    \end{theorem}

    A natural next step in the study of elliptic curves is to classify their torsion structures, focusing on classes closely related to elliptic curves defined over \(\QQ\). This paper focuses on elliptic curves with rational \(j\)-invariant, which are quadratic twists of elliptic curves defined over \(\QQ\). This approach was first explored by T. Gu\v{z}i\'{c} in \cite{guzvic}, where he classified the torsion subgroups of elliptic curves \(E\) with \(j(E) \in \QQ\), defined over extensions \(K/\QQ\) of prime degree. Similarly, J.E. Cremona and F. Najman in \cite{Cremona_Najman} classified, up to isomorphism, the torsion subgroups of elliptic curves \(E\) with \(j(E) \in \QQ\), defined over extensions \(K/\QQ\) of degree coprime to primes less than 11.

    The main result of this paper is an analogue of Theorem \ref{fujita_list} for elliptic curves with rational \(j\)-invariants. Specifically, we prove the following theorem:

    \begin{theorem}[Main Theorem]\label{main Theorem}
    Let \(E\) be an elliptic curve defined over \(\QQ(2^{\infty})\) satisfying \(j(E) \in \QQ\setminus\{0,1728\}\). Then, the torsion subgroup \(E(\QQ(2^{\infty}))_{\tors}\) is isomorphic to one of the following groups\(:\)
   \begin{align*}
     \ZZ&/N_1\ZZ, &\hspace{1cm}  & N_1 = 1,3,5,7,9,13,15,\\
     \ZZ/2\ZZ &\oplus \ZZ/2N_2\ZZ,     & \hspace{1cm}   & N_2 = 1,\; \ldots, 6, 8,\\
      \ZZ/3\ZZ &\oplus \ZZ/3\ZZ,\\
    \ZZ/4\ZZ &\oplus \ZZ/4N_4\ZZ, &\hspace{1cm}  & N_4 = 1,2,3,4,\\
    \ZZ/2N_5\ZZ &\oplus \ZZ/2N_5\ZZ, &\hspace{1cm}  & N_5 = 3,4,6,\\
    \ZZ/6\ZZ &\oplus \ZZ/12\ZZ.
        \end{align*}
    \end{theorem}

    The proof follows this approach. Let \(E\) be an elliptic curve defined over \(\QQ(2^\infty)\), with \(j(E) \in \QQ\setminus\{0,1728\}\). As we will see in Section \ref{lemmas on elliptic curves}, there exists an elliptic curve \(E'\) defined over \(\QQ\) and a quadratic extension \(L/\QQ(2^\infty)\) such that \(E\) is isomorphic to \(E'\), over \(L\). This relationship allows us to infer properties of the field \(L\) based on the existence of torsion points in \(E(\QQ(2^{\infty}))_{\tors}\). For instance, if \(E(\QQ(2^{\infty}))_{\tors}\) contains a point of order 5, then \(L\) is a subfield of \(\QQ^{\ab}\), the maximal abelian extension of \(\QQ\), and we can use the classification obtained by M. Chou in \cite{CHOU2} to limit the possible torsion subgroups \(E(\QQ(2^\infty))_{\tors}\).
    
      \begin{theorem}[M. Chou \cite{CHOU2}, Theorem 1.2]\label{chou_list_maximalabelian}
        Let \(E\) be an elliptic curve defined over \(\QQ\). Then, \(E(\QQ^{\ab})_{\tors}\) is isomorphic to one of the following groups\(:\)
        \begin{align*}
     &\ZZ/N_1\ZZ,     & \hspace{0.5cm}   & N_1 = 1, 3, 5, 7, 9, 11, 13, 15, 17, 19, 21, \newline 25, 27, 37, 43, 67, 163,  \\
    &\ZZ/2\ZZ \oplus \ZZ/2N_2\ZZ, &\hspace{0.5cm}  & N_2 = 1, 2, \;\ldots, 9,\\
    &\ZZ/3\ZZ \oplus \ZZ/3N_3\ZZ, &\hspace{0.5cm}  & N_3 = 1, 3,\\
    &\ZZ/4\ZZ \oplus \ZZ/4N_4\ZZ, &\hspace{0.5cm}  & N_4 = 1, 2, 3, 4,\\
    &\ZZ/5\ZZ \oplus \ZZ/5\ZZ,\\
    &\ZZ/6\ZZ \oplus \ZZ/6\ZZ,\\
    &\ZZ/8\ZZ \oplus \ZZ/8\ZZ.
        \end{align*}
    \end{theorem}
    
    Similarly, we can prove that if \(E(\QQ(2^{\infty}))_{\tors}\) contains a point of order 4, then \(L\) is contained in \(\QQ(\D_4^{\infty})\), the compositum of all Galois extensions \(K/\QQ\) for which the Galois group \(\gal(K/\QQ)\) is isomorphic to the dihedral group \(\D_4\). In this case, we use the classification of torsion subgroup structures over \(\QQ(\D_4^{\infty})\) obtained by H. B. Daniels in \cite{Daniels}.

     \begin{theorem}[H. B. Daniels \cite{Daniels}, Theorem 1.13]\label{Daniels_list}
    Let \(E\) be an elliptic curve defined over \(\QQ\). Then, \(E(\QQ(\D_4^{\infty}))_{\tors}\) is isomorphic to one of the following groups\(:\)
        \begin{align*}
     \ZZ&/N_1\ZZ,     & \hspace{0.5cm}   & N_1 = 1,3,5,7,9,13,15,  \\
    \ZZ/3\ZZ &\oplus \ZZ/3N_3\ZZ, &\hspace{0.5cm}  & N_3 = 1,5,\\
    \ZZ/4\ZZ &\oplus \ZZ/4N_4\ZZ, &\hspace{0.5cm}  & N_4 = 1,\; \ldots, 6, 8,\\
    \ZZ/5\ZZ &\oplus \ZZ/5\ZZ,\\
    \ZZ/8\ZZ &\oplus \ZZ/8N_8\ZZ, &\hspace{0.5cm}  & N_8 = 1, 2, 3, 4,\\
    \ZZ/12\ZZ &\oplus \ZZ/12N_{12}\ZZ, &\hspace{0.5cm}  & N_{12} = 1, 2,\\
    \ZZ/16\ZZ &\oplus \ZZ/16\ZZ.
        \end{align*}
    \end{theorem}

\begin{remark}\label{remark main theorem}
    It is known that all the groups listed in Theorem \ref{main Theorem}, with the exceptions of \((\mathbb{Z}/6\mathbb{Z}) \oplus (\mathbb{Z}/12\mathbb{Z})\) and \((\mathbb{Z}/12\mathbb{Z}) \oplus (\mathbb{Z}/12\mathbb{Z})\), occur as the torsion subgroup \(E(\mathbb{Q}(2^{\infty}))_{\text{tors}}\) for some elliptic curve \(E/\mathbb{Q}(2^{\infty})\) with \(j\)-invariant \(j(E) \in \mathbb{Q} \setminus \{0, 1728\}\). In Table 1, we provide the LMFDB labels (see \cite{LMFDB}) for elliptic curves exhibiting each of the torsion structures listed in Theorem \ref{main Theorem}, with the aforementioned two exceptions. Some of these examples are taken from \cite{Daniels}, \cite{Fujita_1}, \cite{Fujita_2}, and \cite{Laska_Lorenz}.
\end{remark}

\subsection*{Outline of the paper}  

In Section \ref{elliptic curves over Q}, we present known results regarding elliptic curves defined over \(\QQ\) and establish several lemmas concerning the extension degree of the field of definition of torsion points. In Section \ref{lemmas on elliptic curves}, we review the theory of elliptic curves with rational \(j\)-invariant and prove various lemmas that relate their torsion points to those of elliptic curves defined over \(\QQ\). As a further preparation, in Section \ref{lemmas on group theory}, we prove two essential lemmas concerning group theory. Finally, in Section \ref{section main}, we establish Theorem \ref{main Theorem}, beginning with a study of the \(p\)-primary part \(E(\QQ(2^\infty))[p^{\infty}]\) for each prime \(p\), followed by an analysis of their compatibility.

    \begin{table}[h!]
\begin{center}
\renewcommand{\arraystretch}{1}
\begin{tabular}{|l|l||l|l|}\hline
	\(E/F\) & \(E(F)_{\tors}\) & \(E/F\) & \(E(F)_{\tors}\)\\
	\hline
	\href{http://www.lmfdb.org/EllipticCurve/Q/26b2}{\texttt{26b2}}     & $\{\mathcal{O}\}$ &
     \href{https://www.lmfdb.org/EllipticCurve/Q/150/c/3}{\texttt{150a3}}   & $\ZZ/2\ZZ\oplus \ZZ/10\ZZ$ \\
	\href{http://www.lmfdb.org/EllipticCurve/Q/19a2}{\texttt{19a2}}     & $\ZZ/3\ZZ$ &
        \href{https://www.lmfdb.org/EllipticCurve/Q/90/c/7}{\texttt{90c3}}     & $\ZZ/2\ZZ\oplus \ZZ/12\ZZ$ \\
     \href{https://www.lmfdb.org/EllipticCurve/Q/11/a/2}{\texttt{11a1}}     & $\ZZ/5\ZZ$ &
      \href{http://www.lmfdb.org/EllipticCurve/Q/210e1}{\texttt{210e1}}     & $\ZZ/2\ZZ\oplus \ZZ/16\ZZ$ \\
	\href{http://www.lmfdb.org/EllipticCurve/Q/26b1}{\texttt{26b1}}     & $\ZZ/7\ZZ$ &
	\href{https://www.lmfdb.org/EllipticCurve/Q/19a1/}{\texttt{19a1}}     & $\ZZ/3\ZZ\oplus \ZZ/3\ZZ$ \\
	\href{http://www.lmfdb.org/EllipticCurve/Q/54a2}{\texttt{54a2}}     & $\ZZ/9\ZZ$ &
	\href{https://www.lmfdb.org/EllipticCurve/Q/40/a/4}{\texttt{40a4}}     & $\ZZ/4\ZZ\oplus \ZZ/4\ZZ$ \\
	\href{https://www.lmfdb.org/EllipticCurve/2.2.17.1/100.1/e/2}{\texttt{100.1-e2}} & $\ZZ/13\ZZ$ &
	\href{http://www.lmfdb.org/EllipticCurve/Q/15a2}{\texttt{15a2}}   & $\ZZ/4\ZZ\oplus \ZZ/8\ZZ$ \\
	\href{https://www.lmfdb.org/EllipticCurve/Q/50/b/3}{\texttt{50b1}}     & $\ZZ/15\ZZ$ &
	\href{https://www.lmfdb.org/EllipticCurve/Q/30/a/6}{\texttt{30a2}}     & $\ZZ/4\ZZ\oplus \ZZ/12\ZZ$ \\
      \href{https://www.lmfdb.org/EllipticCurve/Q/46a1/}{\texttt{46a1}}     & $\ZZ/2\ZZ\oplus \ZZ/2\ZZ$ &
	\href{https://www.lmfdb.org/EllipticCurve/Q/210/e/6}{\texttt{210e2}}     & $\ZZ/4\ZZ\oplus \ZZ/16\ZZ$ \\
       \href{https://www.lmfdb.org/EllipticCurve/Q/96/b/4}{\texttt{96a4}}   & $\ZZ/2\ZZ\oplus \ZZ/4\ZZ$ &
       \href{https://www.lmfdb.org/EllipticCurve/Q/14/a/6}{\texttt{14a1}} & $\ZZ/6\ZZ\oplus \ZZ/6\ZZ$ \\
	\href{https://www.lmfdb.org/EllipticCurve/Q/20/a/3}{\texttt{20a2}}   & $\ZZ/2\ZZ\oplus \ZZ/6\ZZ$ &
	\href{https://www.lmfdb.org/EllipticCurve/Q/15/a/5}{\texttt{15a1}}     & $\ZZ/8\ZZ\oplus \ZZ/8\ZZ$ \\
   	 \href{https://www.lmfdb.org/EllipticCurve/Q/48/a/3}{\texttt{48a3}}   & $\ZZ/2\ZZ\oplus \ZZ/8\ZZ$ & &\\
	\hline
\end{tabular}
\end{center}
\caption{Examples of elliptic curves for each torsion structure over \(\QQ(2^{\infty})\), except for the groups \((\ZZ/6\ZZ) \oplus (\ZZ/12\ZZ)\) and \((\ZZ/12\ZZ) \oplus (\ZZ/12\ZZ)\).}  \label{tab:Examples}
\end{table}

\subsection*{Acknowledgments}
This paper was part of the author's master's thesis. He is deeply grateful to everyone who supported him throughout that period. He is especially grateful to Yuichiro Taguchi and Yuto Nagashima for their valuable discussions and insightful feedback on the initial draft of this paper. He also sincerely thanks the anonymous referee for the critical and detailed review, which significantly improved the manuscript.

\section{Torsion points of elliptic curves over \(\QQ\)}\label{elliptic curves over Q}

    In this section, we review known results concerning torsion points of elliptic curves defined over \(\QQ\). 
    In the course of proving Theorem \ref{main Theorem}, it will be crucial to determine the extension degree of the field of definition of a point of order \(p\). The main result is the following theorem:

\begin{theorem}[E. Gonz\'{a}lez-Jimen\'{e}z and F. Najman \cite{GonzalezJimenez2016GrowthOT}, Proposition 5.7]\label{exact_degree}
Let \(E\) be an elliptic curve defined over \(\QQ\), \(p\) a prime number and \(P \in E(\overline{\QQ})\) a point of order \(p\). Then, for \(p\leq 13\) or \(p=37\) all of the cases in the table below occur, and they are the only possible ones. For \(p = 17\), the cases in the table below are the only possible ones \((\)it is still unknown if the cases \([\QQ(P):\QQ] = 96\) or \(192\) might occur\()\).
\[\begin{array}{|c|c|}
\hline
p & [\QQ(P):\QQ]\\
\hline
2 & 1,2,3\\
\hline
3 & 1,2,3,4,6,8\\
\hline
5 & 1,2,4,5,8,10,16,20,24\\
\hline
7 & 1,2,3,6,7,9,12,14,18,21,24,36,42,48\\
\hline
11 & 5,10,20,{ 40},55,{ 80},100,110,120\\
\hline
13 & 3,4,6,12,{ 24},39,{ 48},52,72,78,96,{ 144},156,168\\
\hline
17 & 8, 16, 32, {\bf 96}, 136, {\bf 192}, 256, 272, 288\\
\hline
37 & 12,36,{ 72},444,{ 1296},1332, 1368\\
\hline
\end{array}\]
\end{theorem}

Unfortunately, stronger results are required for the proof of Theorem \ref{main Theorem}. A useful method for further restricting the possibilities of the extension degree of the field of definition of a torsion point is the concept of cyclic \(n\)-isogeny.

\begin{definition}
    Let \(E\) be an elliptic curve defined over a field \(K\) and let \(\overline{K}\) be an algebraic closure of \(K\). We say that \(E\) admits a (cyclic) \(n\)-isogeny over \(K\) if there exists an elliptic curve \(\widetilde{E}/K\) and an isogeny \(\phi: E \to \widetilde{E}\) defined over \(K\), such that the kernel of \(\phi\), \(\ker(\phi)\), is a cyclic subgroup of \(E(\overline{K})\) of order \(n\) that is invariant under the action of \(\gal(\overline{K}/K)\).
\end{definition}

\begin{remark}
    In the preceding definition, when \(K = \QQ\), we simply refer to an \(n\)-isogeny over \(\QQ\) as an \(n\)-isogeny, omitting the phrase {\it over \(\QQ\)} for simplicity.
\end{remark}

In this paper, we frequently rely on the following well-known equivalence, which allows us to interpret the existence of a cyclic \(n\)-isogeny over \(K\) in terms of the Galois-invariant cyclic subgroups of \(E(\overline{K})_{\tors}\).

\begin{lemma}[Equivalent condition for cyclic isogeny]
    Let \(E\) be an elliptic curve defined over a field \(K\) of characteristic zero. Then \(E\) admits a cyclic \(n\)-isogeny over \(K\) if and only if there exists a Galois extension \(\widetilde{K}\) of \(K\) and a cyclic subgroup \(C_n \subset E(\widetilde{K})_{\tors}\) of order \(n\), such that \(C_n\) is invariant under the action of \(\gal(\widetilde{K}/K)\).
\end{lemma}

     For elliptic curves defined over \(\QQ\), we have a complete list of the possible \(n\)-isogenies.

    \begin{theorem}[B. Mazur \cite{Mazur1978}. See \cite{robledo_field_of_definition} for a complete list of references]\label{isogeny_list}
    Let \(E\) be an elliptic curve defined over \(\QQ\). If \(E\) has an \(n\)-isogeny over \(\QQ\), then \(n \leq 19\) or \(n \in \{21, 25, 27, 37, 43, 67, 163\}\). If \(E\) does not have complex multiplication, then \(n \leq 18\) or \(n \in \{21, 25, 37\}\).
    \end{theorem}  

    M. Chou established a straightforward criterion for an elliptic curve to possess an \(n\)-isogeny in \cite{CHOU}.

    \begin{lemma}[M. Chou. Adapted from the proof of Proposition 4.1 in \cite{CHOU}]\label{cyclic_isogeny_when_cn+cnm}
    Let \(K\) be a Galois extension of \(\QQ\), and let \(E\) be an elliptic curve defined over \(\QQ\) such that \(E(K)[nm] = (\ZZ/nm\ZZ) \oplus (\ZZ/m\ZZ)\) for some positive integers \(n\) and \(m\). Then, \(E\) has an \(n\)-isogeny.
    \end{lemma}

    \begin{proof}
        This follows from the fact that \([m]\left(E(K)[nm]\right)\) is a \(\gal(K/\QQ)\)-invariant cyclic subgroup of \(E(K)\) of order \(n\).
    \end{proof}

    Using the concept of cyclic isogeny, we now prove two lemmas that provide additional restrictions on the extension degree of the field of definition of torsion points of elliptic curves defined over \(\QQ\). The first lemma follows directly from the definition of cyclic isogeny.

    \begin{lemma}\label{degree of definition isogeny}
         Let \(E\) be an elliptic curve defined over \(\QQ\), and \(K\) a Galois extension of \(\QQ\). If \(P_n \in E(K)\) is a point of order n such that the subgroup \(\left<P_n\right>\) generated by \(P_n\) is \(\gal(K/\QQ)\)-invariant, then \([\QQ(P_n):\QQ]\) divides \(\varphi(n)\).
    \end{lemma}

    \begin{proof}
        Consider the homomorphism induced by the action of \(\gal(K/\QQ)\) on \(\left<P_n\right>\):
        \begin{align*}
           \Phi\colon \gal(K/\QQ) &\longrightarrow \left(\ZZ/n\ZZ\right)^{\times}\\
                  \sigma &\longmapsto a_{\sigma},
        \end{align*}
        where \(P_n^{\sigma} = a_{\sigma}P_n\). By the isomorphism theorem, we observe that \(\gal(\QQ(P_n)/\QQ) \cong \gal(K/\QQ)/\Ker(\Phi)\) is isomorphic to a subgroup of \(\left(\ZZ/n\ZZ\right)^{\times}\). Therefore, \([\QQ(P_n):\QQ] = \# \gal(\QQ(P_n)/\QQ)\) divides \(\varphi(n) = \# (\ZZ/n\ZZ)^\times\).
    \end{proof}

    The second lemma provides a stronger result for the case where the elliptic curve \(E\) has a \(p^{n-1}\)-isogeny of a particular form.
    
    \begin{lemma}\label{degree of definition isogeny + rational torsion}
        Let \(E\) be an elliptic curve defined over \(\QQ\), and let \(K\) be a finite Galois extension of \(\QQ\) such that \(E(K)[p^n] = \left<P_{p^n},Q_p\right> \cong (\ZZ/p^n\ZZ) \oplus (\ZZ/p\ZZ) \), where \(p\) is a prime number, \(n\) is a positive integer, \(P_{p^n} \in E(K)[p^n] \setminus E(\QQ)\) is a point of order \(p^n\), and \(Q_p \in E(\QQ)[p]\) is a point of order \(p\). If the extension degree \([K:\QQ]\) is coprime to \(p\), then \([\QQ(P_{p^n}):\QQ]\) divides \(p-1\). 
    \end{lemma}

    \begin{proof}
        Consider the following map:
        \begin{align*}
            \Phi \colon \gal(K/\QQ) &\longrightarrow \left(\ZZ/p^n\ZZ\right)^{\times}\\
                  \sigma &\longmapsto a_{\sigma},
        \end{align*}
        where \(P_{p^n}^{\sigma} = a_{\sigma}P_{p^n} + b_{\sigma}Q_p\). We claim that \(\Phi\) is a well-defined group homomorphism. First, we show that \(a_{\sigma}\) is indeed coprime to \(p\). Suppose, for contradiction, that \(a_{\sigma}\) is divisible by \(p\). There are two cases to consider:
        \begin{enumerate}
        \item[(i)] If \(n = 1\), then \(P_p^{\sigma} = b_{\sigma} Q_p\), which implies that \(P_p \in E(\QQ)\). However, this contradicts the assumption that \(P_p \notin E(\QQ)\).
        \item[(ii)] If \(n \geq 2\), then \(P_{p^n}^\sigma\), hence \(P_{p^n}\), would not have order \(p^n\), which contradicts the definition of \(P_{p^n}\).
        \end{enumerate}
        Thus, \(a_{\sigma}\) must be coprime to \(p\).
        
        Next, we verify that \(\Phi\) is a group homomorphism. Let \(\sigma, \tau \in \gal(K/\QQ)\). Then,
        \[
        P_{p^n}^{\sigma\tau} = (P_{p^n}^\sigma)^\tau = (a_{\sigma}P_{p^n} + b_{\sigma}Q_p)^{\tau} = a_{\tau}a_{\sigma}P_{p^n} + (a_{\sigma}b_{\tau} + b_{\sigma})Q_p.
        \]
        Thus, \(a_{\sigma\tau} = a_{\sigma}a_{\tau}\), showing that \(\Phi\) is a group homomorphism.

        Now, consider the kernel of \(\Phi\):
        \[
        \Ker(\Phi) = \left\{ \sigma \in \gal(K/\QQ) \;|\; P_{p^n}^{\sigma} = P_{p^n} + b_{\sigma}Q_p \right\}.
        \]
        Let \(\sigma \in \Ker(\Phi)\). Composing the action of \(\sigma\) for \(\order(\sigma)\) times, we get
        \[
        P_{p^n} = P_{p^n}^{\sigma^{\order(\sigma)}} = P_{p^n} + \order(\sigma)b_{\sigma} Q_p.
        \]
        Therefore, \([\order(\sigma)b_{\sigma}] Q_p = \OO\), meaning that \(\order(\sigma)b_{\sigma}\) must be divisible by \(p\). Since \(\order(\sigma)\) is coprime to \(p\), we conclude that \(b_{\sigma}\) must be divisible by \(p\). Hence the kernel of \(\Phi\) is given by 
        \[
        \Ker(\Phi) = \left\{ \sigma \in \gal(K/\QQ)\; |\; P_{p^n}^{\sigma} = P_{p^n} \right\}.
        \]

        Finally, by Galois theory, we have 
        \[
        \gal(\QQ(P_{p^n})/\QQ) \cong \gal(K/\QQ)/\Ker(\Phi),
        \]
        which, by the isomorphism theorem, is isomorphic to a subgroup of \(\left(\ZZ/p^n\ZZ\right)^{\times}\). Since \(\varphi(p^n) = p^{n-1}(p-1)\) and \([K : \QQ]\) is coprime to \(p\), we conclude that \([\QQ(P_{p^n}) :\QQ]\) divides \(p-1\).
    \end{proof}

\section{Elliptic curves with rational \(j\)-invariant}\label{lemmas on elliptic curves}

    In the present section, we briefly review the theory of elliptic curves with rational \(j\)-invariant and prove some necessary results for the proof of Theorem \ref{main Theorem}.
    
    Let \(K\) be a field of characteristic zero. We know that every elliptic curve \(E\) defined over \(K\) can be represented by the form
 \[
 E : y^2 = x^3 + Ax + B\;\;\;\; (A, B \in K).
 \]
 Suppose its \(j\)-invariant \(j(E)\) satisfies \(j(E) \in \QQ\setminus\{0,1728\}\), and define \(E'\) to be the elliptic curve
\begin{equation}\label{eq: rational quadratic twist}
        E' \colon y^2  = x^3 - \frac{27j(E)}{j(E) - 1728}x + \frac{54j(E)}{j(E) - 1728}.
\end{equation}
Observe that \(E'\) is defined over \(\QQ\) and satisfies \(j(E') = j(E)\). For simplicity, denote the coefficients of \(E'\) by 
\begin{align*}
    A' &\coloneq - \frac{27j(E)}{j(E) - 1728},\\
    B' &\coloneq  \frac{54j(E)}{j(E) - 1728}.
\end{align*}
Since \(j(E) = j(E')\) is not 0 or 1728, by the definition of the \(j\)-invariant, all these coefficients \(A,B,A'\) and \(B'\) are non-zero. It is known (see J. H. Silverman \cite{Silverman:1338326}) that \(E'\) is a quadratic twist of \(E\) by \(D = AB'/BA' \in K\), that is \[E' = E^D:\; y^2 = x^3 + AD^2x + BD^3,\] and there exists an isomorphism \(\phi\), defined over the field \(L_{K,D} \coloneq K(\sqrt{D})\), as follows:
  \begin{align}\label{eq: isomorphism}
  \begin{split}
      \phi\colon E &\overset{\sim}{\longrightarrow} E' \\
             (x,y) &\longmapsto \left(xD, yD\sqrt{D}\right).
    \end{split}
  \end{align}

  \begin{remark}
      The field \(L_{K,D}\) depends only on the choice of the equation for the elliptic curve \(E\).
  \end{remark}

    We  use the above notation for the rest of this section.

    \begin{notation}
        From now, if \(E\) is an elliptic curve defined over a field \(K\) of characteristic zero, satisfying \(j(E) \in \QQ\setminus\{0,1728\}\), then \(E'\) is the elliptic curve defined by the formula in \eqref{eq: rational quadratic twist}, \(D\) is the element of \(K\) such that \(E' = E^D\), \(L = L_{K,D} = K(\sqrt{D})\) is the field of definition of the isomorphism \(\phi\), as defined in \eqref{eq: isomorphism}, and \(\widehat{L}\) is the Galois closure of \(L\) over \(\QQ\), in case \(L/\QQ\) is an algebraic extension.
    \end{notation}

    From the discussion above, we get a simple lemma about the image by \(\phi\) of points of finite order of \(E\).

    \begin{lemma}\label{remark about field of definition when twisting}
        If \(E(K)\) has a point \(P_n\) of order \(n\), then its image \(\phi(P_n)\) in \(E'(L)\) belongs to \(E'(K)\) if and only if   \(\sqrt{D} \in K\), that is \(L = K\), or \(n = 2\).
    \end{lemma}

    \begin{proof}
        Let \(P_n = (x,y) \in E(K)\). Its image \(\phi(P) = (xD, yD\sqrt{D})\) belongs to \(E'(K)\) if and only if   \(yD\sqrt{D} \in K\), which is equivalent to \(\sqrt{D} \in K\) or \(y = 0\). The latter case means that \(P_n = (x, 0)\) has order \(n = 2\).
    \end{proof}

    The idea of the proof of Theorem \ref{main Theorem} is to relate the properties of elliptic curves defined over the field \(\QQ(2^{\infty})\) with \(j\)-invariant in \(\QQ\setminus\{0,1728\}\), to properties of elliptic curves defined over \(\QQ\). 

    The most general theorem of this nature is the following:
    
    \begin{theorem}[E. Gonz\'alez-Jimen\'ez and J. M. Tornero \cite{gonzalesjimenesalone}, Theorem 3]\label{twist_torsion}
    Let \(E\) be an elliptic curve over a field \(K\) of characteristic zero, \(D\) a non-square element of \(K\), \(L = K(\sqrt{D})\) and \(E^D\) a quadratic twist of \(E\) by \(D\). There exist a pair of homomorphisms 
    \begin{equation*}
        E(L) \overset{\Psi}{\longrightarrow} E(K) \times E^D(K) \overset{\overline{\Psi}}{\longrightarrow} E(L)
    \end{equation*}
    such that \(\overline{\Psi} \circ \Psi = [2]\) and \(\Psi \circ \overline{\Psi} = [2] \times [2]\). 
    
    Moreover, \(\Ker(\Psi) \subset E(L)[2]\), \(\Ker(\overline{\Psi}) \subset E(K)[2] \times E^D(K)[2]\), and the exponent of both groups \(\coker(\Psi)\) and \(\coker(\overline{\Psi})\) are equal to 2.

    Finally, for \(n\) odd, there exists an isomorphism 
    \begin{equation*}
        E(L)[n] \cong E(K)[n] \times E^D(K)[n].
    \end{equation*}
\end{theorem}

    In Section \ref{section main}, we  prove, by contradiction, that in several cases the extension \(L/\QQ\) is Galois. For this we  need the following lemma: 

 \begin{lemma}\label{conjugate_torsion}
    Let \(E\) be an elliptic curve defined over \(\QQ\), \(L\) an extesion of \(\QQ\) and \(\widehat{L}\) its Galois closure over \(\QQ\). For every \(\sigma \in \gal(\widehat{L}/\QQ)\), we have a \textbf{group} isomorphism:
    \begin{align*}
        \nu : E(L) &\overset{\sim}{\longrightarrow} E(L^\sigma) \\
                  (x,y)       &\longmapsto (\sigma(x), \sigma(y)),
    \end{align*}
    where \(L^{\sigma}\) is the image of \(L\) by \(\sigma\).
\end{lemma}

    \begin{proof}
    Since \(E\) is defined over \(\QQ\), the point \((\sigma(x),\sigma(y))\) belongs to \(E(L^\sigma)\), and \(\nu\) is a well-defined map. Since \((P + Q)^{\sigma} = P^{\sigma} + Q^{\sigma}\), \(\nu\) is a group homomorphism. Finally, using the same argument to \(\sigma^{-1}\), we conclude that \(\nu\) is an isomorphism.
\end{proof}

    Next, we  prove some lemmas concerning how points of finite order of \(E(K)\) influence the nature of the field extension \(L/\QQ\). Before that, we fix a notation we  use in the rest of this paper.

    \begin{notation}
        Let \(G\) and \(H\) be groups. For the sake of simplicity, we write \(G \subseteq E(K)\) and \(E(K) \subseteq H\) when \(E(K)\) has a subgroup isomorphic to \(G\) and \(E(K)\) is isomorphic to a subgroup of \(H\), respectively.
    \end{notation}

 \begin{lemma}\label{L galois}
      Let \(E\) be an elliptic curve defined over a Galois extension \(K\) of \(\QQ\), satisfying \(j(E)  \in \QQ\setminus\{0,1728\}\). If \(\ZZ/n\ZZ \subseteq E(K)\) and \((\ZZ/n\ZZ) \oplus (\ZZ/n\ZZ) \subseteq E'(L)\) for some integer \(n \geq 3\), then \(L/\QQ\) is a Galois extension.
  \end{lemma}

  \begin{proof}
        Suppose \(L\neq K\). The second hypothesis is equivalent to \(E'[n] \subseteq E'(L)\). By Lemma \ref{remark about field of definition when twisting}, we know that \(E'[n] \not\subseteq E'(K)\) and, consequently, \(L = K(E'[n])\). This last field is the compositum of Galois extensions of \(\QQ\), so it is itself Galois over \(\QQ\).
  \end{proof}

  \begin{lemma}\label{L hat is of order 16}
      Let \(E\) be an elliptic curve defined over a Galois extension \(K\) of \(\QQ\), satisfying \(j(E)  \in \QQ\setminus\{0,1728\}\). If \(\ZZ/p^n\ZZ \subseteq E(K)\) and \(\ZZ/p\ZZ \not\subseteq E'(K)\) for some odd prime \(p\) and some positive integer \(n\), and the extension \(L/\QQ\) is not Galois, then the Galois closure \(\widehat{L}\) of \(L\) over \(\QQ\) satisfies \([\widehat{L}:K] = 4\), and \((\ZZ/p^n\ZZ) \oplus (\ZZ/p^n\ZZ) \subseteq E'(\widehat{L})\).
  \end{lemma}

    \begin{proof}
        Since the extension \(L/\QQ\) is not Galois and \(E'\) is a quadratic twist of \(E\), the contrapositive of Lemma \ref{L galois}, together with Theorem \ref{twist_torsion}, implies that \(E(K)[p^n] \cong \ZZ/p^n\ZZ\) and \(E'(K)[p^n] = \{\OO\}\). Let \(P_{p^n} \in E(K)\) be a point of order \(p^n\) and \(\phi(P_{p^n})\) its image in \(E'(L)\). Theorem \ref{twist_torsion} implies that \(E'(L)[p^n] = \left<\phi(P_{p^n})\right> \cong \ZZ/p^n\ZZ\). Let \(\sigma \in \gal(\widehat{L}/\QQ)\) be such that \(L^{\sigma} \neq L\). By Lemma \ref{conjugate_torsion}, we see that \(E'(L^{\sigma})[p^n] = \left<\phi(P_{p^n})^{\sigma}\right> \cong \ZZ/p^n\ZZ\). We claim that \(E'[p^n] = \left<\phi(P_{p^n}), \phi(P_{p^n})^{\sigma}\right>\). 
        Indeed, if \(a\) and \(b\) are integers such that \([a]\phi(P_{p^n}) + [b]\phi(P_{p^n})^{\sigma} = \OO\), then both \([a]\phi(P_{p^n})\) and \([b]\phi(P_{p^n})^{\sigma}\) belong to \(E'(L \cap L^{\sigma})[p^n] = E'(K)[p^n] = \{\OO\}\), hence we conclude that \(a \equiv b \equiv 0 \pmod{p^n}\), which proves our claim. Finally, since \(L = K(\phi(P_{p^n}))\) and \(L^{\sigma} = K(\phi(P_{p^n})^{\sigma})\) are quadratic extensions of \(K\), we conclude that  \(\widehat{L} = K(\phi(P_{p^n}),\phi(P_{p^n})^{\sigma}) = K(E'[p^n])\) is a quartic extension of \(K\) and \((\ZZ/p^n\ZZ) \oplus (\ZZ/p^n\ZZ) = E'[p^n] \subseteq E'(\widehat{L})\).
    \end{proof}

    For points of order \(2^n\) we have a slightly different result.

    \begin{lemma}\label{L hat is of order 16 case p = 2}
    Let \(E\) be an elliptic curve defined over a Galois extension \(K\) of \(\QQ\), satisfying \(j(E)  \in \QQ\setminus\{0,1728\}\). If \(\ZZ/2^n\ZZ \subseteq E(K)\) for some integer \(n \geq 2\), \(E'(K)[2^{\infty}] \subseteq (\ZZ/2\ZZ)\oplus(\ZZ/2\ZZ)\) and the extension \(L/\QQ\) is not Galois, then the Galois closure \(\widehat{L}\) of \(L\) over \(\QQ\) satisfies \([\widehat{L}:K] = 4\), and \((\ZZ/2^{n-1}\ZZ) \oplus (\ZZ/2^{n-1}\ZZ) \subseteq E'(\widehat{L})\).
    \end{lemma}

    \begin{proof}
        Let \(P_{2^n} \in E(K)[2^n]\) be a point of order \(2^n\), with image \(\phi(P_{2^n})\) in \(E'(L)\), and let \(\sigma \in \gal(\widehat{L}/\QQ)\) be such that \(L^{\sigma} \neq L\). Lemma \ref{conjugate_torsion} implies that \(E'(L^{\sigma})\) also has a point \(\phi(P_{2^n})^{\sigma}\) of order \(2^n\). Since \(n \geq 2\), we know, by Lemma \ref{remark about field of definition when twisting}, that either of the points \(\phi(P_{2^n})\) and \(\phi(P_{2^n})^{\sigma}\) does not belong to \(E'(K)\). Now, we divide the proof into two cases:
        \begin{enumerate}
            \item[(i)] The first case is when \([2^{n-1}]\phi(P_{2^n}) = [2^{n-1}]\phi(P_{2^n})^{\sigma}\). This equality, together with Lemma \ref{remark about field of definition when twisting}, implies that \(\phi(P_{2^n}) - \phi(P_{2^n})^{\sigma} \in E'(\widehat{L})[2^{n-1}]\) is a point of order \(2^{n-1}\). We claim that 
        \begin{equation}\label{eq: 1, L hat is of order 16, case p = 2}
        E'[2^{n-1}] = \left<[2]\phi(P_{2^n}), \phi(P_{2^n}) - \phi(P_{2^n})^{\sigma}\right>,
        \end{equation}
        which implies that \(\widehat{L} = LL^{\sigma}\), \([\widehat{L}:K] = 4\) and \((\ZZ/2^{n-1}\ZZ) \oplus (\ZZ/2^{n-1}\ZZ) = E'[2^{n-1}] \subseteq E'(\widehat{L})\). Indeed, let \(a\) and \(b\) be integers such that 
        \begin{equation*}
        [a]([2]\phi(P_{2^n})) + [b](\phi(P_{2^n}) - \phi(P_{2^n})^{\sigma}) = \OO.
        \end{equation*}
        Rearranging, we have
        \[
        [2a+b]\phi(P_{2^n}) - [b]\phi(P_{2^n})^{\sigma} = \OO
        \]
        which implies that both \([2a+b]\phi(P_{2^n})\) and \([b]\phi(P_{2^n})^{\sigma}\) belong to \(E'(L\cap L^{\sigma}) = E'(K)\). But since \(E'(K)[2^{\infty}] \subseteq (\ZZ/2\ZZ) \oplus (\ZZ/2\ZZ)\), we need \(b\) to be divisible by \(2^{n-1}\), hence \([2a]\phi(P_{2^n}) = \OO\), and \(a\) is also divisible by \(2^{n-1}\), concluding the proof of the equality \eqref{eq: 1, L hat is of order 16, case p = 2}.

        \item[(ii)] The second case is when \([2^{n-1}]\phi(P_{2^n}) \neq [2^{n-1}]\phi(P_{2^n})^{\sigma}\). Here, we claim that 
        \begin{equation}\label{eq: 2, L hat is of order 16, case p = 2}
        E'[2^n] = \left<\phi(P_{2^n}),\phi(P_{2^n})^{\sigma}\right>
        \end{equation}
        which implies that \(\widehat{L} = LL^{\sigma}\), \([\widehat{L}:K] = 4\) and \((\ZZ/2^{n}\ZZ) \oplus (\ZZ/2^{n}\ZZ) = E'[2^n] \subseteq E'(\widehat{L})\).
        Indeed, let \(a\) and \(b\) be integers such that 
        \begin{equation}\label{eq: linear combination}
        [a]\phi(P_{2^n}) + [b]\phi(P_{2^n})^{\sigma} = \OO.
        \end{equation}
        Here, again, both points \([a]\phi(P_{2^n})\) and \([b]\phi(P_{2^n})^{\sigma}\) belong to \(E'(L\cap L^{\sigma}) = E'(K)\). Hence, by the hypothesis \(E'(K)[2^{\infty}] \subseteq (\ZZ/2\ZZ) \oplus (\ZZ/2\ZZ)\), both \(a\) and \(b\) are divisible by \(2^{n-1}\). If neither \(a\) nor \(b\) are divisible by \(2^n\) then, by the equation \eqref{eq: linear combination}, we get \([2^{n-1}]\phi(P_{2^n}) = [2^{n-1}]\phi(P_{2^n})^{\sigma}\), contradicting our assumption. Hence either \(a\) or \(b\) are divisible by \(2^n\), but if, for instance, \(a\) is divisible by \(2^n\), then \([b]\phi(P_{2^n})^{\sigma} = \OO\) and \(b\) is also divisible by \(2^n\). By symmetry, both \(a\) and \(b\) are divisible by \(2^n\), concluding the proof of the equality \eqref{eq: 2, L hat is of order 16, case p = 2}.
        \end{enumerate}
    \end{proof}

\section{Technical lemmas from group theory}\label{lemmas on group theory}

    As a final preparation for the proof of Theorem \ref{main Theorem}, in this section, we present two necessary lemmas on groups.

    \begin{lemma}\label{group theory lemma 1}
        If \(G\) is a group with normal subgroups \(H_1\) and \(H_2\) where \(H_1\) has order \(2\), \(G/H_1\) has exponent \(2\), and \(G/H_2 \cong \ZZ/4\ZZ\), then \(G\) is abelian.
    \end{lemma}

    \begin{proof}
        Since both the quotient groups \(G/H_1\) and \(G/H_2\) are abelian, the commutator subgroup \([G,G]\) is contained in \(H_1 \cap H_2\), implying that \([G,G]\) has order 1 or 2. Now, if \([G,G]\) has order 2, that is, \( H_1 = [G,G] \subseteq H_2\), then, by the isomorphism theorem, we obtain
        \[
        (G/H_1)/(H_2/H_1) \cong G/H_2 \cong \ZZ/4\ZZ,
        \]
        which is not possible, since \(G/H_1\) only has points of order 2. Therefore, \([G,G] = \left\{e\right\}\) and \(G\) is abelian.
    \end{proof}

    \begin{lemma}\label{group theory lemma 2}
        If \(G\) is a finite group of order \(2^n\), with \(n \geq 2\), having normal subgroups \(H_1\) and \(H_2\) such that \(G/H_1 \cong (\ZZ/2\ZZ)^{\oplus (n-2)}\) and \(G/H_2 \cong \ZZ/4\ZZ\), respectively, then every subgroup of \(H_1\) is normal in \(G\).
    \end{lemma}

    \begin{proof}
        We need to prove this for a non-abelian group \(G\). By the same argument in the proof of the last lemma, it follows that \(H_1 \not\subseteq H_2\) and \([G,G] \subseteq H_1 \cap H_2\). Since \(G\) is non-abelian we conclude that \([G,G]\) has order 2. If \(H_1 \cong \ZZ/4\ZZ\), then its unique non-trivial subgroup is \([G,G]\) which we know that is a normal subgroup of \(G\). From now on, suppose \(H_1 \cong (\ZZ/2\ZZ) \oplus (\ZZ/2\ZZ)\). First, we observe that every element \(x \in G\) has order dividing 4. Indeed, since 
        \[
        G/H_1 \cong (\ZZ/2\ZZ)^{\oplus (n-2)},
        \]
        we have \(x^2 \in H_1\), which implies \(x^4 = e\). In particular, it follows that \(G^2\) is a subset of \(H_1\) and since we can prove, using the structure theorem for finite abelian groups, that 
        \begin{equation}\label{eq: commutator quotient}
        G/[G,G] \cong (\ZZ/2\ZZ)^{\oplus (n-3)} \oplus \ZZ/4\ZZ,
        \end{equation}
        we conclude that \(G^2 \neq [G,G]\), otherwise \(G^2\) would be a normal subgroup of \(G\) implying that \(G/G^2\) is a 2-elementary abelian group, contradicting the isomorphism \eqref{eq: commutator quotient}. Since \(G^2\) has order at least 2, we have \(H_1 = G^2[G,G]\).
        
        Finally, we claim that both \([G,G]\) and \(G^2\) are subsets of \(\Z(G)\), the center of \(G\), concluding that \(H_1 \subseteq \Z(G)\), and every subgroup of \(H_1\) is a normal subgroup of \(G\). Let \([G,G] = \{e,a\}\). Since \([G,G]\) is a normal subgroup of \(G\), for every element \(x \in G\) we have \(x^{-1}ax = a\), where it follows that \([G,G] \subseteq \Z(G)\). 
        
        Now, let \(y^2 \in G^2\). For every \(x \in G\) we have three possibilities: 
        \begin{enumerate}
            \item[(i)] If \([x,y] = e\), then \(x\) commutes with \(y\) and, in particular, with \(y^2\).
            \item[(ii)] If \([y^{-1},x] = e\), then \(x\) commutes with \(y^{-1}\) and, in particular, with \(y^{-2} = y^2\). 
            \item[(iii)] If neither (i) nor (ii) holds, that is, if \([x,y] \neq e\) and \([y^{-1},x] \neq e\), then, since \([G,G]\) has order 2,  
        \[
        [x,y] = [y^{-1},x].
        \]
        By some calculations, we get 
        \[
        [x,y^2] = [x,y]^2 = e,
        \]
        showing that \(x\) commutes with \(y^2\).
        \end{enumerate}
        Since \(y^2 \in G^2\) was chosen arbitrarily, we conclude that \(G^2 \subseteq \Z(G)\).
    \end{proof}

    As an application, we prove a useful property of quadratic extensions of the field \(\QQ(2^{\infty}) = \QQ(\{\sqrt{m} \colon m \in \ZZ\})\), which appears in the proof of Theorem \ref{main Theorem}.

    \begin{lemma}\label{group lemma galois theory}
         Let \(L\) be a quadratic extension of the field \(\QQ(2^{\infty})\), such that \(L/\QQ\) is not a Galois extension, and its Galois closure \(\widehat{L}\) satisfies \([\widehat{L}: \QQ(2^{\infty})] = 4\). Then, there is no intermediate field \(K\) of the extension \(\widehat{L}/\QQ\) such that \(K\) is a quartic cyclic extension of \(\QQ\).
    \end{lemma}

    \begin{proof}
    Suppose there exists a quartic cyclic extension \(K\) of \(\QQ\), contained in \(\widehat{L}\). Let \(\alpha , \beta, \gamma \in \widehat{L}\) be such that \(\widehat{L} = \QQ(2^{\infty})(\alpha)\), \(L = \QQ(2^\infty)(\beta)\) and \(K = \QQ(\gamma)\), respectively. Let \(\alpha^{(i)}\) (\(i \in I\), \(I\) is a finite set) be the conjugates of \(\alpha\) over \(\QQ\) and let
    \begin{align*}
         \alpha^{(i)} &= a_0^{(i)} + a_1^{(i)}\alpha + a_2^{(i)}\alpha^2 + a_3^{(i)}\alpha^3\;\;(a_0^{(i)}, a_1^{(i)}, a_2^{(i)}, a_3^{(i)} \in \QQ(2^{\infty}),\; i \in I)\\
        \beta &=  b_0 + b_1\alpha + b_2\alpha^2 + b_3\alpha^3\;\;(b_0, b_1, b_2, b_3 \in \QQ(2^{\infty}))\\
        \gamma &=  c_0 + c_1\alpha + c_2\alpha^2 + c_3\alpha^3\;\;(c_0, c_1, c_2, c_3 \in \QQ(2^{\infty}))\\
        P(x) &= x^4 + d_3x^3 + d_2x^2 + d_1x + d_0\;\;(d_0, d_1, d_2, d_3 \in \QQ(2^{\infty}))
    \end{align*}
    be, respectively, the representation of \(\alpha^{(i)}\), \(\beta\) and \(\gamma\) in the \(\QQ(2^{\infty})\)-basis \(1, \alpha, \alpha^2, \alpha^3\) of \(\widehat{L}\), and the minimal polynomial of \(\alpha\) over \(\QQ(2^{\infty})\). Consider the field 
    \[F \coloneq \QQ(a_j^{(i)}, b_j, c_j, d_j| i \in I, 0 \leq j \leq 3)\] 
     and its extension \(\widetilde{L} \coloneq F(\alpha)\). By construction, both \(F\) and \(\widetilde{L}\) are Galois extensions of \(\QQ\). Since \(F \subseteq \QQ(2^{\infty})\) and \(\beta \in \widetilde{L}\), it follows that \([\widetilde{L}:F] = 4\) and \(K \subseteq \widetilde{L}\). Applying Lemma \ref{group theory lemma 2} with \(G = \gal(\widetilde{L}/\QQ)\), \(H_1 = \gal(\widetilde{L}/F)\) and \(H_2 = \gal(\widetilde{L}/K)\), we conclude that every intermediate field of the extension \(\widetilde{L}/F\) is Galois over \(\QQ\). In particular, \(F(\beta)/\QQ\) and hence \(L/\QQ\) are Galois extensions, contradicting the assumption in the statement.
    \end{proof}

\section{Proof of the main theorem}\label{section main}

    First, we  establish the notation used throughout this section.
    
    \begin{notation}
    In this section, unless otherwise stated, \(E\)  denotes an elliptic curve defined over the field \(\QQ(2^\infty) = \mathbb{Q}(\{\sqrt{m} \colon m \in \mathbb{Z}\})\), with \(j(E) \in \mathbb{Q} \setminus \{0, 1728\}\). The elliptic curve \(E' = E^D\) is defined by the formula in \eqref{eq: rational quadratic twist}, and \(\phi\) denotes the isomorphism between \(E\) and \(E'\) as specified in \eqref{eq: isomorphism}, both of which are introduced at the beginning of Section \ref{lemmas on elliptic curves}. The isomorphism \(\phi\) is defined over the field \(L = \QQ(2^\infty)(\sqrt{D})\) and \(\widehat{L}\) is the Galois closure of \(L\) over \(\QQ\).
    \end{notation}

    The proof of Theorem \ref{main Theorem} is organized as follows: In Subsections \ref{The case p = 2} through \ref{The case p = 13}, we analyze the \(p\)-primary torsion subgroups \(E(\QQ(2^\infty))[p^{\infty}]\) for each prime \(p\). Then, in Subsection \ref{Miscellaneous cases}, we examine the compatibility of the \(p\)-primary torsion subgroups identified in the previous subsections.

    For the first part, the following proposition establishes that it suffices to consider only the primes 2, 3, 5, 7, and 13.

    \begin{proposition}\label{E[p] list}
        Let \(p\) be a prime number. If \(E(\QQ(2^\infty))\) has a point of order \(p\), then \(p \leq 7\) or \(p = 13\).
    \end{proposition}

    To prove this proposition, we  need the following consequence of the Weil pairing.

     \begin{lemma}\label{cn+cn_contained}
      Let \(E\) be an elliptic curve defined over a field \(K\) of characteristic zero. If \(E(K)_{\tors}\) contains \(E[n]\), then \(K\) contains \(\QQ(\zeta_n)\). In particular, if the extension \(K/\QQ\) is finite, then \(\varphi(n)\) divides \([K:\QQ]\).
  \end{lemma}

    \begin{proof}
    Consider the Weil pairing:
    \begin{equation*}
            e_n \colon E[n] \times E[n] \longrightarrow \mu_n.
    \end{equation*}
     Since \(E(K)\) contains \(E[n]\), there exist points \(S,T \in E(K)\) such that \(e_n(S,T)\) is a primitive \(n\)-th root of unity. By the Galois invariance of \(e_n\), we conclude that \(e_n(S,T)\) belongs to \(K\). Therefore, the field \(\QQ(\zeta_n)\) is contained in \(K\).
    \end{proof}

    \begin{proof}[Proof of Proposition \ref{E[p] list}]
        Suppose \(E(\QQ(2^\infty))\) has a point \(P\) of order \(p \geq 11\). Let \(K\) be the smallest subfield of \(\QQ(2^\infty)\) containing \(\QQ\), the coefficients of \(E\) and the coordinates of \(P\). Since \(K\) is a finite extension of \(\QQ\) contained in \(\QQ(2^\infty)\), \(\gal(K/\QQ)\) is isomorphic to \((\ZZ/2\ZZ)^{r}\) for some positive integer \(r\). In this proof, consider \(E\) to be defined over \(K\), \(L = K(\sqrt{D})\), and the isomorphism \(\phi: E \to E'\) be defined over \(L\).

         We claim that \(L/\QQ\) is a Galois extension. Indeed, if the extension \(L/\QQ\) is not Galois, then Lemma \ref{L hat is of order 16} implies that \([\widehat{L}:K] = 4\) and \(E'[p] \subseteq E'(\widehat{L})\). From Lemma \ref{cn+cn_contained}, we have \(\QQ(\zeta_p) \subseteq \widehat{L}\). Since \(\QQ(\zeta_p)/\QQ\) is a cyclic extension of order \(p-1\), the intersection \(\QQ(\zeta_p) \cap K\) is either \(\QQ\) or a quadratic extension of \(\QQ\). This implies that the extension \(K\QQ(\zeta_p)/\QQ\) has degree \(2^r(p-1)\) or \(2^{r-1}(p-1)\), respectively. Since \(p \geq 11\), we get a contradiction as 
         \begin{equation}\label{eq: ineq for contradiction}
             [K\QQ(\zeta_p):\QQ] \geq 2^{r-1}(p-1) \geq 2^{r-1}\cdot 10 > 2^{r+2} = [\widehat{L}:\QQ].
         \end{equation}
         
        Now, we claim that \(E'\) has a \(p\)-isogeny. To prove this, we observe that \(E'[p] \not\subseteq E'(L)\). Indeed, if \(E'[p] \subseteq E'(L)\), then by Lemma \ref{cn+cn_contained}, we would have \(\QQ(\zeta_p) \subseteq L\), which is a cyclic extension of \(\QQ\) of degree \(p-1 \geq 10\), leading to a contradiction similar to the inequality \eqref{eq: ineq for contradiction}. Hence, \(E'(L)[p] \cong \ZZ/p\ZZ\), and by Lemma \ref{cyclic_isogeny_when_cn+cnm}, \(E'\) has a \(p\)-isogeny.

        Next, examining the list in Theorem \ref{isogeny_list}, we find that \(p\) must be equal to either 11, 13, 17, 19, 37, 43, 67, or 163. By Lemma \ref{degree of definition isogeny}, the extension \(\QQ(\phi(P))/\QQ\) is cyclic of degree dividing \(p-1\). To avoid a contradiction similar to the one in \eqref{eq: ineq for contradiction}, the degree \([\QQ(\phi(P)):\QQ]\) must be equal to either \(1, 2,\) or \(4\).

        If \(p \in \{11, 19, 43, 67\}\), then \(p-1\) is not divisible by 4, implying that \(\QQ(\phi(P)) \subseteq \QQ(2^\infty)\). By Lemma \ref{remark about field of definition when twisting}, we have an isomorphism \(E(\QQ(2^\infty))_{\tors} \cong E'(\QQ(2^\infty))_{\tors}\). However, Theorem \ref{fujita_list} states that \(E'(\QQ(2^\infty))_{\tors}\) contains no points of order 11, 19, 43, or 67, leading to a contradiction.

        If \(p \in \{17, 37\}\), then Theorem \ref{exact_degree} implies that \([\QQ(\phi(P)):\QQ] \geq 8\). Thus, the only remaining possibility is \(p = 13\).
    \end{proof}

    \begin{remark}\label{remark of E[p] list}
    Note that, as proved in the preceding proof, if \(E(\mathbb{Q}(2^\infty))\) has a point of order 13, then \(E'(\QQ(2^\infty))[13] = \{\mathcal{O}\}\), \(L/\mathbb{Q}\) is a Galois extension, and \(E'\) has a 13-isogeny. This fact will be used in the proof of Lemma \ref{13 L abelian}.

    \end{remark}

    Now, we study \(E(\QQ(2^\infty))[p^{\infty}]\) when \(p\) is 2, 3, 5, 7 or 13.

\subsection{The 2-primary part of \(E(\QQ(2^\infty))_{\tors}\)}\label{The case p = 2}\hfill\\

    In this subsection, we  classify the possibilities for \(E(\QQ(2^\infty))[2^{\infty}]\) up to isomorphism. More specifically, we  prove the following proposition:

    \begin{proposition}\label{classification 2-elem torsion}
        \(E(\QQ(2^\infty))[2^{\infty}]\) is isomorphic to a subgroup of \((\ZZ/8\ZZ)\oplus(\ZZ/8\ZZ)\) or is isomorphic to \((\ZZ/4\ZZ)\oplus(\ZZ/16\ZZ)\).
    \end{proposition}

    The proof will be established through several lemmas. Before proving the necessary lemmas, we  recall some properties of 2-torsion points and introduce some notation. As stated in Lemma \ref{remark about field of definition when twisting}, points of order 2 have the following property: \(P_2 \in E(\QQ(2^\infty))[2]\) if and only if   \(\phi(P_2) \in E'(\QQ(2^\infty))[2]\). Together with the definition of \(\QQ(2^\infty)\), this yields the following unique property of 2-torsion points.

    \begin{lemma}\label{lemma: full 2-torsion}
        If \(E(\QQ(2^\infty))[2] \neq \{\mathcal{O}\}\), then \(E(\QQ(2^\infty))[2] = E[2]\) and \(E'(\QQ)[2] \neq \{\mathcal{O}\}\). 
    \end{lemma}

    \begin{proof}
        Recall that \(E'\) has the form 
        \[
        E': y^2 = x^3 + A'x + B' \;\; (A',B' \in \QQ).
        \]
        A point \(P = (x,y) \in E'(\QQ(2^\infty))\) has order 2 if and only if  \(y = 0\) and \(x\) is a root of the cubic polynomial \(x^3 + A'x + B'\). If \(E(\QQ(2^\infty))[2] \neq \{\mathcal{O}\}\), then, by Lemma \ref{remark about field of definition when twisting}, \(E'(\QQ(2^\infty))[2] \neq \{\mathcal{O}\}\) and \(x^3 + A'x + B'\) has a root in \(\QQ(2^\infty)\). If \(x^3 + A'x + B'\) is irreducible in \(\QQ[x]\) then the extension degree of the field of definition of the point of order 2 in \(E'(\QQ(2^\infty))\) is 3, which is not possible since \(\QQ(2^\infty)\) contains no subfield of degree 3 over \(\QQ\). This implies that \(x^3 + A'x + B'\) has a root in \(\QQ\) and the other roots also belong to \(\QQ(2^\infty)\). Therefore, \(E'(\QQ(2^\infty))[2] = E'[2]\), and cosequently, by Lemma \ref{remark about field of definition when twisting}, \(E(\QQ(2^\infty))[2] = E[2]\).
    \end{proof}

    This lemma allows us to use the following theorem from Knapp \cite{knapp}.

    \begin{theorem}[Knapp \cite{knapp}, Theorem 4.2]\label{propknapp}
    Let \(K\) be a field of characteristic not equal to 2 or 3, and let \(E\) be an elliptic curve defined over \(K\) given by \(y^2 = (x-\alpha)(x-\beta)(x -\gamma)\) with \(\alpha, \beta, \gamma \in K\). For \(P = (x,y) \in E(K)\), there exists a \(K\)-rational point \(Q = (x',y') \in E(K)\) such that \([2]Q = P\) if and only if   \(x - \alpha, x - \beta, x - \gamma\) are all squares in \(K\). Moreover, if we fix each sign of \(\sqrt{x - \alpha}\), \(\sqrt{x - \beta}\) and \(\sqrt{x - \gamma}\), then \(x'\) equals one of the following:
    \[
    \sqrt{x - \alpha}\sqrt{x - \beta} \pm \sqrt{x - \alpha}\sqrt{x - \gamma} \pm \sqrt{x - \beta}\sqrt{x - \gamma} + x
    \]
    or
    \[
    -\sqrt{x - \alpha}\sqrt{x - \beta} \pm \sqrt{x - \alpha}\sqrt{x - \gamma} \mp \sqrt{x - \beta}\sqrt{x - \gamma} + x.
    \]
    \end{theorem}

    The first case we need to verify is when \(E'(\QQ)\) contains all the points of order 2.

    \begin{lemma}\label{all 2-tor rational}
    Suppose \(E(\QQ(2^\infty))\) has a point of order \(4\). If \(E'[2] \subseteq E'(\QQ)\) then \(E'[4] \subseteq E'(\QQ(2^\infty))\) and \(L = \QQ(2^\infty)\). In particular, \(E(\QQ(2^\infty))_{\tors}\) is isomorphic to one of the groups in Theorem \ref{fujita_list} and \(E(\QQ(2^\infty))[2^{\infty}]\) satisfies the conclusion of Proposition \ref{classification 2-elem torsion}.
    \end{lemma}

    \begin{proof}
        The hypothesis is equivalent to
        \[
        E': y^2 = x^3 + A'x + B' = (x - \alpha)(x - \beta)(x - \gamma) \;\;\; (\alpha, \beta, \gamma \in \QQ).
        \]
        This implies that \(\sqrt{\alpha - \beta}, \sqrt{\alpha - \gamma}\), and \(\sqrt{\beta - \gamma}\) are in \(\QQ(2^\infty)\). By Theorem \ref{propknapp}, \(E'(\QQ(2^\infty))\) contains \(E'[4]\) and, by Lemma \ref{remark about field of definition when twisting}, \(L = \QQ(2^\infty)\).
    \end{proof}

    From this point onward, we  consider the case \(E'(\QQ)[2] \cong \ZZ/2\ZZ\), and for this purpose, we introduce some notation.

    \begin{notation}\label{notation E'(Q)[2] = C2}
    When \(E'(\QQ)[2] \cong \ZZ/2\ZZ\), \(E'\) can be represented, after applying a translation by an element of \(\QQ\), by the form
    \begin{equation}\label{eq: 2-torsion rep}
    E': y^2 = x(x - (a + b\sqrt{n}))(x - (a - b\sqrt{n})),
    \end{equation}
    where \(a,b \in \QQ\) and \(n\) is a square-free integer. From now, we suppose \(E'\) has the above representation.
    \end{notation}

    Using this notation we can prove the following two lemmas:

    \begin{lemma}\label{c2 + c4 L galois}
        If \(E(\QQ(2^\infty))\) has a point of order \(4\) and \(E'(\QQ)[2] \cong \ZZ/2\ZZ\), then \(\widehat{L}\) is contained in \(\QQ(\D_4^{\infty})\), the compositum of all Galois extensions \(K/\QQ\) with Galois group \(\gal(K/\QQ)\) isomorphic to \(\D_4\).
    \end{lemma}

    \begin{proof}
        Let \(P_4 \in E(\QQ(2^\infty))\) be a point of order 4 and \(\phi(P_4) = (x,y)\) its image in \(E'(L)\). By Lemma \ref{remark about field of definition when twisting}, we know that either \(L = \QQ(2^\infty)\), which implies that \(L \subseteq \QQ(\D_4^{\infty})\), or \(L = \QQ(2^\infty)(\phi(P_4))\). By Theorem \ref{propknapp}, we know that either one of the following conditions is satisfied (using Notation \ref{notation E'(Q)[2] = C2}):
        \begin{itemize}
            \item both \(\sqrt{-a - b\sqrt{n}}\) and \(\sqrt{-a + b\sqrt{n}}\) are in \(L\), or
            \item both \(\sqrt{a + b\sqrt{n}}\) and \(\sqrt{2b\sqrt{n}}\) are in \(L\), or
            \item both \(\sqrt{a - b\sqrt{n}}\) and \(\sqrt{-2b\sqrt{n}}\) are in \(L\).
        \end{itemize}
        Hence, either \(L = \QQ(2^\infty)(\sqrt{\sqrt{n}})\), \(L = \QQ(2^\infty)(\sqrt{a + b\sqrt{n}})\) or \(L = \QQ(2^\infty)(\sqrt{a - b\sqrt{n}})\). Observe that the extension \(K = \QQ(\sqrt{n})(\sqrt{\alpha + \beta\sqrt{n}})/\QQ\), for \(\alpha, \beta \in \QQ\), has Galois group \(\gal(\widehat{K}/\QQ)\) isomorphic to either \((\ZZ/2\ZZ)\oplus(\ZZ/2\ZZ)\), \(\ZZ/4\ZZ\), or \(\D_4\). By \cite[Proposition 3.4]{Daniels_errata}, we conclude that \(K\), and, consequently, \(L\) and \(\widehat{L}\) are contained in \(\QQ(\D_4^{\infty})\).
    \end{proof}

     \begin{lemma}\label{c8 + c8 L=F}
        If \((\ZZ/8\ZZ) \oplus (\ZZ/8\ZZ) \subseteq E(\QQ(2^\infty))\) and \(E'(\QQ)[2] \cong \ZZ/2\ZZ\), then \(L = \QQ(2^\infty)\).
     \end{lemma}

     \begin{proof}
         Suppose that \(L \neq \QQ(2^\infty)\). By Lemma \ref{L galois}, we know that \(L/\mathbb{Q}\) is a Galois extension, and by the proof of Lemma \ref{c2 + c4 L galois}, we have
\[
L = \QQ(2^\infty)(\sqrt{\sqrt{n}}) = \QQ(2^\infty)(\sqrt{a + b\sqrt{n}}).
\]

Since \(E'[4] \subseteq E'(L)\), Theorem \ref{propknapp} implies that the \(x\)-coordinates of the points of order 4 in \(E'(L)\) are one of the following:
\begin{align*}
    &\pm \sqrt{a^2 - b^2n}, \\
    &\pm \sqrt{a + b\sqrt{n}}\sqrt{2b\sqrt{n}} + (a + b\sqrt{n}), \\
    &\pm \sqrt{a - b\sqrt{n}}\sqrt{-2b\sqrt{n}} + (a - b\sqrt{n}).
\end{align*}
Since \(E'(L)\) has 12 points of order 4, all six possibilities listed above indeed occur as \(x\)-coordinates of some point of order 4. Using these \(x\)-coordinates, the assumption that \(E'[8] \subseteq E'(L)\), and Theorem \ref{propknapp}, we obtain that
\[
\sqrt{\sqrt{a + b\sqrt{n}}}\sqrt{\sqrt{2b}}
\]
belongs to \(L\). Let
\[
\beta \coloneq \sqrt{a + b\sqrt{n}}\sqrt{2b},
\]
and observe that \(\beta^2 \in \QQ(2^\infty)\), \(\beta \notin \QQ(2^\infty)\), and \(L = \QQ(2^\infty)(\beta) = \QQ(2^\infty)(\sqrt{\beta})\). We claim that this is not possible. Indeed, suppose that there exist \(a', b' \in \QQ(2^\infty)\) such that \(\sqrt{\beta} = a' + b'\beta\). Squaring both sides, and using the properties of \(\beta\) derived above, we obtain \(a' = 1/2b'\). Finally, squaring both sides of \(\sqrt{\beta} = 1/2b' + b'\beta\), we get
\[
\beta = \pm\frac{\sqrt{-1}}{2b'^2} \in \QQ(2^\infty),
\]
which contradicts the properties of \(\beta\).
     \end{proof}

     With the results from the lemmas in this subsection, along with the lists of Theorems \ref{fujita_list} and \ref{Daniels_list}, we have shown that \(E(\QQ(2^\infty))[2^\infty]\) is either isomorphic to a subgroup of \((\ZZ/8\ZZ) \oplus (\ZZ/8\ZZ)\) or isomorphic to a subgroup of \((\ZZ/4\ZZ) \oplus (\ZZ/32\ZZ)\). Thus, to complete the proof of Proposition \ref{classification 2-elem torsion}, it remains to prove the non-existence of points of order 32.

    \begin{lemma}\label{no c32}
    \(E(\QQ(2^\infty))\) has no point of order \(32\).
    \end{lemma}

    \begin{proof}
        Suppose that \(E(\QQ(2^\infty))\) has a point of order 32. We claim that the extension \(L/\QQ\) is Galois. Suppose it is not. By Lemmas \ref{L galois} and \ref{lemma: full 2-torsion}, \(E'(L)[2^{\infty}] \cong (\ZZ/2\ZZ)\oplus(\ZZ/32\ZZ)\) and, by Lemma \ref{L hat is of order 16 case p = 2},  \(E'[16] \subseteq E'(\widehat{L})\). Now, we can apply Lemma \ref{cn+cn_contained} and get that \(\widehat{L}\) contains the cyclotomic field \(\QQ(\zeta_{16})\), which contains a quartic cyclic extension of \(\QQ\). This contradicts Lemma \ref{group lemma galois theory} since \(L/\QQ\) is not a Galois extension by assumption.

        Now, when \(L/\QQ\) is a Galois extension we have three possibilities: \(E'(L)[2^{\infty}]\) is isomorphic to \((\ZZ/2^k\ZZ)\oplus(\ZZ/32\ZZ)\) for some \(k \in \{1,2,3\}\). Let \(P_{32} \in E(\QQ(2^\infty))\) be a point of order 32, and \(\phi(P_{32})\) its image in \(E'(L)\). By the structure of \(E'(L)[2^{\infty}]\) we conclude that the subgroup \(\left<[8]\phi(P_{32})\right>\) is \(\gal(L/\QQ)\)-invariant and its field of definition \(\QQ([8]\phi(P_{32}))\) is an abelian extension of \(\QQ\). Since \(L = \QQ(2^\infty)([8]\phi(P_{32}))\), we observe that \(L/\QQ\) is an abelian extension and, in particular, \(L\) is contained in \(\QQ^{\ab}\), the maximal abelian extension of \(\QQ\). From this we get a contradiction, because \(E'(\QQ^{\ab})\) has no point of order 32, by the list in Theorem \ref{chou_list_maximalabelian}.
    \end{proof}

    Together, Lemmas \ref{all 2-tor rational}, \ref{c2 + c4 L galois}, \ref{c8 + c8 L=F} and \ref{no c32} complete the proof of Proposition \ref{classification 2-elem torsion}.

\subsection{The 3-primary part of \(E(\QQ(2^\infty))_{\tors}\)}\label{The case p = 3}\hfill\\

    In this subsection, we study the 3-primary torsion subgroups \(E(\QQ(2^{\infty}))[3^{\infty}]\). We prove that a point of order 9 in \(E(\QQ(2^{\infty}))\) makes \(E\) be \(\QQ(2^{\infty})\)-isomorphic to the elliptic curve \(E'\).

    \begin{proposition}
    If \(E(\QQ(2^{\infty}))\) has a point of order \(9\), then \(L = \QQ(2^{\infty})\). Therefore \(E(\QQ(2^{\infty}))_{\tors}\) is isomorphic to one of the groups in Theorem \(\ref{fujita_list}\).
    \end{proposition}

    \begin{proof}
        Suppose that \(E(\QQ(2^{\infty}))\) has a point of order 9. We have two possibilities:
        \begin{enumerate}
            \item[(i)] If \(E'(\QQ(2^{\infty}))[3] = \{\OO\}\). Since \(L\) does not contain the field \(\QQ(\zeta_9)\), Theorem \ref{twist_torsion} and Lemma \ref{cn+cn_contained} imply that \(E'(L)[9] \subseteq (\ZZ/3\ZZ) \oplus (\ZZ/9\ZZ)\). We need to consider the following possibilities: 
            \begin{itemize}
                \item If \(L/\QQ\) is a Galois extension. Then, using Lemma \ref{cyclic_isogeny_when_cn+cnm}, we observe that \(E'\) has a 3-isogeny and if \(P_3 \in E'(L)\) is a point of order 3 such that \(\left<P_3\right>\) is \(\gal(L/\QQ)\)-invariant, Lemma \ref{degree of definition isogeny} implies that \([\QQ(P_3):\QQ]\) divides 2. Hence \(P_3 \in \QQ(2^{\infty})\) and \(L = \QQ(2^{\infty})\). 
                \item If \(L/\QQ\) is not a Galois extension. Then, by Lemma \ref{L hat is of order 16}, we see that \([\widehat{L}:\QQ(2^{\infty})] = 4\) and \(\ZZ/9\ZZ \oplus \ZZ/9\ZZ \subseteq E'(\widehat{L})\). However, Lemma \ref{cn+cn_contained} concludes that \(\QQ(\zeta_9) \subseteq \widehat{L}\), which is not true.
            \end{itemize}

            \item[(ii)] If \(E'(\QQ(2^{\infty}))[3] = \left<P_3\right> \cong \ZZ/3\ZZ\). Then, by Lemma \ref{L galois}, \(L/\QQ\) is a Galois extension and \(E'(L)[9] \cong (\ZZ/3\ZZ) \oplus (\ZZ/9\ZZ)\). Let \(P_9 \in E(\QQ(2^{\infty}))\) be a point of order 9, and \(\phi(P_9)\) its image in \(E'(L)\). Lemma \ref{cyclic_isogeny_when_cn+cnm} implies that \(\left<[3]\phi(P_9)\right>\) is \(\gal(L/\QQ)\)-invariant, and Lemma \ref{degree of definition isogeny} concludes that \([\QQ([3]\phi(P_9)):\QQ]\) divides 2. Hence, by Lemma \ref{remark about field of definition when twisting}, \(L = \QQ(2^{\infty})([3]\phi(P_9)) = \QQ(2^{\infty})\).
        \end{enumerate}
    \end{proof}

\subsection{The 5-primary part of \(E(\QQ(2^\infty))_{\tors}\)}\label{The case p = 5}\hfill\\

    In this subsection, we classify the 5-primary torsion subgroups \(E(\QQ(2^{\infty}))[5^{\infty}]\) up to isomorphism. Similarly to the approach in Subsection \ref{The case p = 2}, we  prove that if \(E(\QQ(2^{\infty}))\) has a point of order 5, then the field \(L\) is contained in \(\mathbb{Q}^{\text{ab}}\), the maximal abelian extension of \(\mathbb{Q}\). Additionally, we  use Theorem \ref{chou_list_maximalabelian} to further restrict the possible structures.
   
   \begin{proposition}\label{5 L abelian}
        If \(E(\QQ(2^{\infty}))\) has a point of order \(5\), then \(L/\QQ\) is an abelian extension.
    \end{proposition}

    \begin{proof}
        Suppose \(L/\QQ\) is not a Galois extension. By Lemmas \ref{L galois}, \ref{L hat is of order 16} and \ref{cn+cn_contained}, we see that \([\widehat{L}:\QQ(2^{\infty})] = 4\) and \(\QQ(\zeta_5) \subseteq \widehat{L}\). This contradicts Lemma \ref{group lemma galois theory}, since \(\QQ(\zeta_5)/\QQ\) is a quartic cyclic extension. 
        
        Now, suppose that \(L \neq \QQ(2^\infty)\). We claim that \(L/\QQ\) has an intermediate field \(K\) such that \(K/\QQ\) is a quartic cyclic extension. Indeed, if \(E'(\QQ(2^\infty))[5] \neq \{\OO\}\), then, by Theorem \ref{twist_torsion} and Lemma \ref{cn+cn_contained}, \(L\) contains \(\QQ(\zeta_5)\), which is a quartic cyclic extension of \(\QQ\). Now, suppose that \(E'(\QQ(2^\infty))[5] = \{\OO\}\) and let \(P_5 \in E(\QQ(2^\infty))[5]\) be a point of order 5. We claim that \(\QQ(\phi(P_5))/\QQ\) is a quartic cyclic extension. Indeed, since \(\left<\phi(P_5)\right>\) is \(\gal(L/\QQ)\)-invariant and, by Lemma \ref{remark about field of definition when twisting}, \([\QQ(\phi(P_5)):\QQ]\) is not equal to neither 1 or 2, we have the following isomorphism:
        \begin{align*}
    \gal(\QQ(\phi(P_5))/\QQ) &\overset{\sim}{\longrightarrow} (\ZZ/5\ZZ)^{\times}\\
            \tau &\longmapsto a_{\tau} \;\;(\phi(P_5)^{\tau} = a_{\tau}\phi(P_5)).
    \end{align*}
        Since \((\ZZ/5\ZZ)^\times\) is isomorphic to \(\ZZ/4\ZZ\), it follows that \(\QQ(\phi(P_5))/\QQ\) is a quartic cyclic extension. 
        
        Finally, applying Lemma \ref{group theory lemma 1} with \(G = \gal(L/\QQ)\), \(H_1 = \gal(L/\QQ(2^\infty))\) and \(H_2 = \gal(L/K)\), we conclude that \(\gal(L/\QQ)\) is an abelian group.
    \end{proof}  

    \begin{proposition}
    \(E(\QQ(2^{\infty}))\) has no point of order \(25\).
    \end{proposition}
    
    \begin{proof}
    Suppose that \(E(\QQ(2^{\infty}))\) has a point of order 25. By Proposition \ref{5 L abelian}, the extension \(L/\QQ\) is abelian. Additionally, by Lemma \ref{cn+cn_contained}, it follows that \(E[25] \not\subseteq E(L)\). This implies that \(E(L)[25] \subseteq (\ZZ/25\ZZ) \oplus (\ZZ/5\ZZ)\). Thus, we have two possible cases:

    \begin{enumerate}
        \item[(i)] If \(E'(\QQ(2^{\infty}))[5] = \{\OO\}\), then by Lemma \ref{cyclic_isogeny_when_cn+cnm} and Theorem \ref{twist_torsion}, \(E'\) has a 25-isogeny. Let \(P_{25} \in E(\QQ(2^\infty))\) be a point of order 25 and \(\phi(P_{25})\) its image in \(E'(L)\). We know that \(\left< \phi(P_{25}) \right>\) is \(\gal(L/\QQ)\)-invariant. By Lemma \ref{degree of definition isogeny}, the extension degree \([\QQ(\phi(P_{25})):\QQ]\) must divide \(\varphi(25) = 20\). Additionally, since \(\QQ(\phi(P_{25}))/\QQ\) is a Galois extension, it follows that \[[\QQ(\phi(P_{25})): \QQ(\phi(P_{25})) \cap \QQ(2^\infty)] = [L:\QQ(2^\infty)] = 2,\] hence \([\QQ(\phi(P_{25})):\QQ]\) is a power of 2. This implies that \([\QQ(\phi(P_{25})):\QQ]\)  must divide 4, which is impossible, as \cite[Corollary 8.7]{GonzalezJimenez2016GrowthOT} states that \(\Phi_{\QQ}(4)\) contains no group of order divisible by 25.
        
        \item[(ii)] If \(E'(\QQ(2^{\infty}))[5] = \left<P_5\right> \cong \ZZ/5\ZZ\) for a point \(P_5\) of order 5, then we claim that \([\QQ(P_5):\QQ] = 1\) or \(2\). Indeed, since \(\QQ(2^{\infty})/\QQ\) is an abelian extension, \(\left<P_5\right>\) is \(\gal(\QQ(2^\infty)/\QQ)\)-invariant and Lemma \ref{degree of definition isogeny} implies that \([\QQ(P_5):\QQ]\) divides 4. Additionally, if it is equal to 4, then \(\gal(\QQ(P_5)/\QQ)\) is isomorphic to \((\ZZ/5\ZZ)^{\times}\) where we get a contradiction, since \(\QQ(P_5)\) is contained in \(\QQ(2^{\infty})\) but \((\ZZ/5\ZZ)^{\times}\) is a cyclic group of order 4. 

        If \([\QQ(P_5):\QQ] = 2\), then \(\QQ(P_5) = \QQ(\sqrt{n})\) for some \(n \in \QQ\). Consider the quadratic twist \(E'^n\) of \(E'\) by \(n\). \(E'^n\) is also a quadratic twist of \(E\) defined over \(\QQ\) and, by Theorem \ref{twist_torsion}, \(E'^n(\QQ)[5] \cong \ZZ/5\ZZ\). Hence, we assume, without loss of generality, that \(P_5 \in E'(\QQ)[5]\). 
    
    Let \(P_{25} \in E(\QQ(2^{\infty}))[25]\) be a point of order 25 and \(\phi(P_{25})\) its image in \(E'(L)\). Define \(K\) as the smallest subfield of \(\QQ(2^\infty)\) containing \(\QQ\) and the coordinates of \(\phi(P_{25})\), and let \(\widehat{K}\) be its Galois closure over \(\QQ\). Since \(\widehat{K} \subseteq L\), it follows that  \(E'(\widehat{K})[25] = \left<\phi(P_{25}), P_5\right> \cong (\ZZ/25\ZZ) \oplus (\ZZ/5\ZZ)\). Moreover, since \(\widehat{K}/\QQ\) is a finite Galois extension of extension degree coprime to 5, we apply Lemma \ref{degree of definition isogeny + rational torsion} to conclude that \([\QQ(\phi(P_{25})):\QQ]\) divides 4. However, this contradicts the classification of \(\Phi_{\QQ}(4)\) in \cite{CHOU, GonzalezJimenez2016GrowthOT}.
    \end{enumerate}
    \end{proof}

\subsection{The 7-primary part of \(E(\QQ(2^\infty))_{\tors}\)}\label{The case p = 7}\hfill\\

    In this subsection, we examine elliptic curves with points of order 7. By Theorem \ref{exact_degree}, the extension degree of the field of definition of a point of order 7 on an elliptic curve defined over \(\QQ\) is not divisible by 4. This property allows us to establish the following proposition.

    \begin{proposition}\label{lemma5.2}
        If \(E(\QQ(2^{\infty}))\) has a point of order \(7\), then \(L = \QQ(2^{\infty})\). Therefore \(E(\QQ(2^{\infty}))_{\tors}\) is isomorphic to one of the groups in Theorem \(\ref{fujita_list}\).
    \end{proposition}

    \begin{proof}
        Let \(P_7 \in E(\QQ(2^{\infty}))\) be a point of order \(7\), and \(\phi(P_7)\) its image in \(E'(L)\). Since \([L:\QQ(2^{\infty})] \leq 2\), Theorem \ref{exact_degree} implies that \([\QQ(\phi(P_7)):\QQ] = 1\) or \(2\). Hence, \(\phi(P_7) \in E'(\QQ(2^{\infty}))\) and, by Lemma \ref{remark about field of definition when twisting}, \(L = \QQ(2^{\infty})\).
    \end{proof}

\subsection{The 13-primary part of \(E(\QQ(2^\infty))_{\tors}\)}\label{The case p = 13}\hfill\\

    In this subsection, we examine elliptic curves with points of order 13.  

 \begin{proposition}\label{13 L abelian}
        If \(E(\QQ(2^{\infty}))\) has a point of order \(13\), then \(L\) is an abelian extension of \(\QQ\), and \(E(\QQ(2^{\infty}))_{\tors}\) is isomorphic to \(\ZZ/13\ZZ\).
    \end{proposition}

    \begin{proof}
        We already know, by Remark \ref{remark of E[p] list}, that if \(E(\QQ(2^{\infty}))\) has a point of order 13, then \(E'(\QQ(2^{\infty}))[13] = \{\OO\}\), \(L/\QQ\) is a Galois extension and \(E'\) has a 13-isogeny. 
        Let \(P_{13} \in E(\QQ(2^{\infty}))\) be a point of order 13 and \(\phi(P_{13}) = (x,y)\) its image in \(E'(L)\). Remark \ref{remark of E[p] list} implies that \(\left<\phi(P_{13})\right>\) is \(\gal(L/\QQ)\)-invariant, and, by Lemma \ref{degree of definition isogeny}, we observe that \([\QQ(\phi(P_{13})):\QQ]\) divides 12. Since, by the definition of \(\phi\), \(x \in \QQ(2^{\infty})\) and, by the defining equation of an elliptic curve, the degree of the extension \(\QQ(x,y)/\QQ(x)\) is at most 2, we observe that \(\QQ(\phi(P_{13}))/\QQ\) is a Galois extension of degree at most 4 and, in particular, an abelian extension.

        Finally, since \(L = \QQ(2^{\infty})(\phi(P_{13})) \subseteq \QQ^{\ab}\), \(E(\QQ(2^{\infty}))_{\tors}\) is a subgroup of \(E'(\QQ^{\ab})_{\tors}\), which is one of the groups in Theorem \ref{chou_list_maximalabelian}, where we observe that the only possibility is it to be isomorphic to \(\ZZ/13\ZZ\).
    \end{proof}

\subsection{Compatibility of \(p\)-Primary Torsion Subgroups in \(E(\QQ(2^{\infty}))_{\text{tors}}\)}\label{Miscellaneous cases}\hfill\\

    Thus far, we have examined the \(p\)-primary torsion subgroups \(E(\QQ(2^{\infty}))[p^{\infty}]\) for each prime \(p\). In this subsection, we investigate which combinations of \(p\)-primary torsion subgroups can occur simultaneously. Specifically, for each prime \(p\), let \(G_p\) denote one of the groups that can appear as a \(p\)-primary torsion subgroup for some elliptic curve. Our goal is to determine whether there exists an elliptic curve \(E/\QQ(2^{\infty})\) such that the torsion subgroup \(E(\QQ(2^{\infty}))_{\text{tors}}\) is isomorphic to the direct product \(\prod_p G_p\). Fortunately, using the stronger lemmas established in Subsections \ref{The case p = 2} through \ref{The case p = 13}, along with Theorems \ref{fujita_list}, \ref{Daniels_list}, and \ref{chou_list_maximalabelian}, it suffices to focus on the case where \(E(\QQ(2^{\infty}))\) has points of order 24.

    \begin{example}
    Let us show, for instance, that \(E(\QQ(2^{\infty}))_{\tors}\) does not contain a point of order 30. Assume, for contradiction, that \(E(\QQ(2^{\infty}))_{\tors}\) contains a point \(P_{30}\) of order 30. Then, \([6]P_{30}\) is a point of order 5. By Proposition \ref{5 L abelian}, this implies that \(L/\QQ\) is an abelian extension. Consequently, \(E(\QQ(2^{\infty}))_{\tors}\) is a subgroup of \(E'(\QQ^{\ab})_{\tors}\).
    However, by Theorem \ref{chou_list_maximalabelian}, the torsion subgroup \(E'(\QQ^{\mathrm{ab}})_{\tors}\) does not contain a point of order 30. 
\end{example}

    \begin{proposition}\label{no point of order 24}
         \(E(\QQ(2^{\infty}))\) has a no point of order \(24\).
    \end{proposition}
    
    \begin{proof}
        Suppose \(E(\QQ(2^{\infty}))\) has a point of order 24. By Lemma \ref{c2 + c4 L galois}, \(L\) is contained in \(\QQ(\D_4^{\infty})\), and, by Theorem \ref{Daniels_list}, we have three possibilities:
        \begin{enumerate}
            \item[(i)] If \(E'(\QQ(\D_4^{\infty}))_{\tors} \cong (\ZZ/4\ZZ)\oplus(\ZZ/24\ZZ)\). By  \cite[Section 7]{Daniels}, \(E'\) has a 3-isogeny, which implies, by Lemmas \ref{degree of definition isogeny} and \ref{remark about field of definition when twisting}, that \(L = \QQ(2^{\infty})\). By Theorem \ref{fujita_list}, \(E(\QQ(2^{\infty}))_{\tors}\) has no point of order 24.

            \item[(ii)] If \(E'(\QQ(\D_4^{\infty}))_{\tors} \cong (\ZZ/8\ZZ)\oplus(\ZZ/24\ZZ)\). By \cite[Section 7]{Daniels}, \(E'\) has its full 2-torsion defined over \(\QQ\). From Theorem \ref{propknapp}, we observe that \(E'\) has its full 4-torsion defined over \(\QQ(2^{\infty})\). Hence, by Lemma \ref{remark about field of definition when twisting}, we conclude that \(L = \QQ(2^{\infty})\). Again, by Theorem \ref{fujita_list}, \(E(\QQ(2^{\infty}))_{\tors}\) has no point of order 24.

            \item[(iii)] If \(E'(\QQ(\D_4^{\infty}))_{\tors} \cong (\ZZ/12\ZZ)\oplus(\ZZ/24\ZZ)\). By \cite[Section 7]{Daniels}, \(j(E') = 8000\). Therefore, we can take the elliptic curve \href{https://www.lmfdb.org/EllipticCurve/Q/256a1/}{\texttt{256a1}} (using the LMFDB notation, see \cite{LMFDB}) as \(E'\), and, observing that the field of definition of a subgroup isomorphic to \((\ZZ/2\ZZ)\oplus(\ZZ/4\ZZ)\) contains a quartic cyclic extension of \(\QQ\), we get, by Lemma \ref{group theory lemma 1}, that \(L/\QQ\) is an abelian extension and, in particular, is contained in \(\QQ^{\ab}\). By Theorem \ref{chou_list_maximalabelian}, \(E'(\QQ^{\ab})\) has no point of order 24.
        \end{enumerate}
    \end{proof}

\end{document}